\documentclass[oneside,english,american]{amsart}
\usepackage[T1]{fontenc}
\usepackage[utf8]{inputenc}
\usepackage{geometry}
\usepackage{float}
\usepackage{amstext}
\usepackage{amsthm}
\usepackage{amssymb}
\usepackage{graphicx}

\makeatletter
\numberwithin{equation}{section}
\numberwithin{figure}{section}
\theoremstyle{plain}
\newtheorem{thm}{\protect\theoremname}
  \theoremstyle{plain}
  \newtheorem{lem}[thm]{\protect\lemmaname}
  \theoremstyle{remark}
  \newtheorem{rem}[thm]{\protect\remarkname}
  \theoremstyle{plain}
  \newtheorem{prop}[thm]{\protect\propositionname}
  \theoremstyle{plain}
  \newtheorem{conjecture}[thm]{\protect\conjecturename}

\date{\today}\usepackage{babel}
\usepackage{babel}

\usepackage{babel}
\providecommand{\conjecturename}{Conjecture}
  \providecommand{\lemmaname}{Lemma}
  \providecommand{\propositionname}{Proposition}
  \providecommand{\remarkname}{Remark}
\providecommand{\theoremname}{Theorem}

\makeatother

\usepackage{babel}
  \addto\captionsamerican{\renewcommand{\conjecturename}{Conjecture}}
  \addto\captionsamerican{\renewcommand{\lemmaname}{Lemma}}
  \addto\captionsamerican{\renewcommand{\propositionname}{Proposition}}
  \addto\captionsamerican{\renewcommand{\remarkname}{Remark}}
  \addto\captionsamerican{\renewcommand{\theoremname}{Theorem}}
  \addto\captionsenglish{\renewcommand{\conjecturename}{Conjecture}}
  \addto\captionsenglish{\renewcommand{\lemmaname}{Lemma}}
  \addto\captionsenglish{\renewcommand{\propositionname}{Proposition}}
  \addto\captionsenglish{\renewcommand{\remarkname}{Remark}}
  \addto\captionsenglish{\renewcommand{\theoremname}{Theorem}}
  \providecommand{\conjecturename}{Conjecture}
  \providecommand{\lemmaname}{Lemma}
  \providecommand{\propositionname}{Proposition}
  \providecommand{\remarkname}{Remark}
\providecommand{\theoremname}{Theorem}

\begin{document}

\title[Polynomials with rational generating functions]{On the zeros of polynomials generated by rational functions with a hyperbolic polynomial type denominator}

\author{Tam\'as Forg\'acs\\
 Khang Tran}
\begin{abstract}
This paper investigates the location of the zeros of a sequence of
polynomials generated by a rational function with a denominator of
the form $G(z,t)=P(t)+zt^{r}$, where the zeros of $P$ are positive and
real. We show that every member of a family of such generating functions
- parametrized by the degree of $P$ and $r$ - gives rise to a sequence
of polynomials $\{H_{m}(z)\}_{m=0}^{\infty}$ that is eventually hyperbolic.
Moreover, when $P(0)>0$ the real zeros of the polynomials $H_{m}(z)$
form a dense subset of an interval $I\subset\mathbb{R}^{+}$, whose
length depends on the particular values of the parameters in the generating
function.  
\end{abstract}

\maketitle

\section{Introduction}

Consider the generating relation 
\[
\sum_{k=0}^{\infty}H_{m}(z)t^{m}=G(z,t)
\]
for a sequence of polynomials $\{H_{m}(z)\}_{m=0}^{\infty}$. For
certain specific choices of the function $G(z,t)$, one can derive
various properties of the polynomials $H_{m}(z)$ from those of the
function $G(z,t)$, including their degrees and their location of
zeros. For example, well-known properties of the classical orthogonal
polynomials\footnote{i.e. the Hermite-, Laguerre-, and the Legendre (more generally the
Jacobi) polynomials} can be obtained from by manipulating their generating functions to
produce recurrence relations satisfied by these polynomials. The recurrence
relations are in turn used to establish the orthogonality of the polynomials
over a certain interval in $\mathbb{R}$ with respect to some weight
function \cite[Ch. 10, 11, 12, 16]{rainville}. In addition to orthogonality,
one also discovers that the generated set of polynomials is \textit{simple},
that is, $\deg H_{m}(z)=m$ for all $m\geq0$. It is through these
connections that one can assert the reality of the zeros of each of
the polynomials in the generated sequence (see for example \cite[p.149, Theorem 55]{rainville}),
along with the fact that the zeros must lie in the interval on which
the family of polynomials is orthogonal. \\
 \indent The goal of the present paper is to describe the extent
to which similar conclusions can be drawn for polynomials generated
by functions of the form 
\[
(\dag)\qquad G(z,t)=\frac{1}{P(t)+zt^{r}},
\]
where $P\in\mathbb{R}[t]$ is a polynomial of degree $n$ with only
positive (real) zeros, and $r$ is a positive integer. We note that
if $n=2$ and $r=1$, a result by the second author shows that he
zeros of the generated polynomials are all real (see \cite[\S 2, Theorem 1, p.\,331]{tran}).
The main result of the present paper\footnote{For pairs $(n,r)\neq(1,1)$ or $(2,1)$}
(cf. Theorem \ref{maintheorem}) is that from a certain point on in
the sequence, polynomials generated by functions of the type $(\dag)$
will all have only real zeros, all of which are located in a particular
interval in $\mathbb{R}$. We find this result appealing in that it
closely resembles the long-established analogous conclusions about
the classical orthogonal polynomials, despite us not knowing explicitly
what the generated sequence of polynomials looks like. \\
 \indent Given that the degree of the denominator of $G(z,t)$ in
$t$ is $d=\max\{n,r\}$, we are assured that the terms of the generated
sequence $\{H_{m}(z)\}_{m=0}^{\infty}$ satisfy an $(d+1)$-term recurrence
relation. We could not, however, find a way to ascertain that they
satisfy a three-term recurrence relation\footnote{a necessary condition for the sequence $\{H_{m}(z)\}_{m=0}^{\infty}$
to be orthogonal (see \cite[Theorem 57, p.151]{rainville})}. Thus the techniques described above for obtaining orthogonality
relations for the generated polynomials do not readily lend themselves
to the solution of our problem. Moreover, the set of polynomials we
obtain using such functions are, in general, not simple. In light
of these obstructions, we construct a proof which establishes that
each $H_{m}(z)$ has at least as many distinct real zeros as its degree.
We accomplish this by first proving that there exists a continuous,
real valued, strictly increasing function $z(\theta)$ on $(0,\pi/r)$,
when viewed as a function of the argument of a non-real zero of $P(t)+zt^{r}$
(Lemma \ref{lem:zfunctheta} and Lemma \ref{lem:zonto}). We then
construct continuous functions $H(\theta;m)$ on $(0,\pi/r)$ with
the property that $H(\theta;m)=0$ if and only if $H_{m}(z(\theta))=0$.
Finally, we demonstrate that for all large $m$, $H(\theta;m)$ has
at least as many zeros on $(0,\pi/r)$ as the degree of $H_{m}(z)$
(Proposition \ref{prop:zerosQ} and Proposition \ref{prop:signchangegeneral}).
Since each of these zeros will give rise to a unique real zero of
$H_{m}(z)$, we obtain that $H_{m}(z)$ must in fact have only real
zeros for all $m\gg1$. The identification of the interval containing
the zeros of $H_{m}(z)$ for large $m$ is accomplished in Lemma \ref{lem:existenceinterval}
and Lemma \ref{lem:zonto}. The interval we obtain is optimal in the
sense that the set 
\[
\mathcal{Z}=\bigcup_{m\gg1}\{z\ |\ H_{m}(z)=0\}
\]
is dense there. \\
 \indent With this outline in hand, we now present the main result
of the paper. 
\begin{thm}
\label{maintheorem} Suppose $P(t)$ is a real polynomial whose zeros
are positive real numbers and $P(0)>0$. Let $r$ be a positive integer
such that $\max\{\deg P,r\}>1$. Let $t_{a}$ and $t_{b}$ be the
smallest positive, and the largest nonpositive real zeros of the polynomial
${\displaystyle {\frac{d}{dx}(-P(t)/t^{r})}}$ respectively. For all
large integers $m$, the zeros of the polynomial $H_{m}(z)$ generated
by 
\begin{equation}
\sum_{m=0}^{\infty}H_{m}(z)t^{m}=\frac{1}{P(t)+zt^{r}}\label{eq:genrel}
\end{equation}
lie in the interval $(a,b)$, where $a=-P(t_{a})/t_{a}^{r}$ and 
\[
b=\begin{cases}
-P(t_{b})/t_{b} & \text{ if }t_{b}\ne0\\
\infty & \text{ if }t_{b}=0
\end{cases}.
\]
Moreover, the set $\mathcal{Z}=\bigcup_{m\gg1}\{z\ |\ H_{m}(z)=0\}$
is dense on $(a,b)$. 
\end{thm}
We remark that the choice $P(t)=(1-t)^{n}$ reproduces Theorem 1 in
\cite{ft} via the equations 
\[
rt^{r-1}P(t)-t^{r}P'(t)=t^{r-1}(1-t)^{n-1}((n-r)t+r),
\]
and 
\[
(t_{b},t_{a})=\begin{cases}
(0,1) & \text{if}\ r>1,n>1\\
{\displaystyle {\left(-\frac{1}{n-1},1\right)}} & \text{if}\ 1=r<n\\
{\displaystyle {\left(0,\frac{r}{r-1}\right)}} & \text{if}\ 1=n<r
\end{cases}.
\]

The following lemma justifies the descriptions of $t_{a}$ and $t_{b}$
in the statement of Theorem \ref{maintheorem}. We refer the reader
to Figure \ref{fig:zfunct} for an illustration of $-P(t)/t$ and
the values $a,b,t_{a}$, and $t_{b}$. 
\begin{lem}
\label{lem:existenceinterval} Let $P(t)=\sum_{k=0}^{n}a_{k}t^{k}=|a_{n}|\prod_{k=1}^{n}(\tau_{k}-t)$
be a polynomial of degree $n$ whose zeros, $\tau_{1}\le\tau_{2}\le\cdots\le\tau_{n}$,
are positive real. The zeros of the polynomial 
\[
R(t)=t^{2r}\frac{d}{dx}(-P(t)/t^{r}))=rt^{r-1}P(t)-t^{r}P'(t)
\]
are real. Furthermore, if $r=1$, then $R(t)$ has a unique negative
zero. 
\end{lem}
\begin{proof}
Let $\gamma_{1},\gamma_{2},\ldots,\gamma_{n-1}$ be the zeros of $P'(t)$.
Note that the zeros of $P(t)$ and $P'(t)$ interlace, that is 
\[
\tau_{1}\le\gamma_{1}\le\tau_{2}\le\gamma_{2}\le\cdots\le\gamma_{n-1}\le\tau_{n}.
\]
Thus, by the Intermediate Value Theorem, each non-trivial interval
$(\tau_{k},\gamma_{k})$, $1\le k<n$, contains a zero of $R(t)$.
If $r>1$, then $R(t)$ has a zero at $0$ with multiplicity $r-1$,
and since common zeros of $P(t)$ and $P'(t)$ are also zeros of $R(t)$,
by degree considerations we see that all zeros of $R(t)$ are real
(and non-negative).\\
 If $r=1$, the facts that $\lim_{t\rightarrow-\infty}P(t)-tP'(t)=-\infty$
and $P(0)>0$ imply that $R(t)$ has at least one negative real zero.
On the other hand, $R(t)$ has $n-1$ non-negative real zeros by the
above considerations. Since $\deg R(t)=n$, we conclude that $R(t)$
has exactly one negative real zero. 
\end{proof}
\begin{figure}
\includegraphics{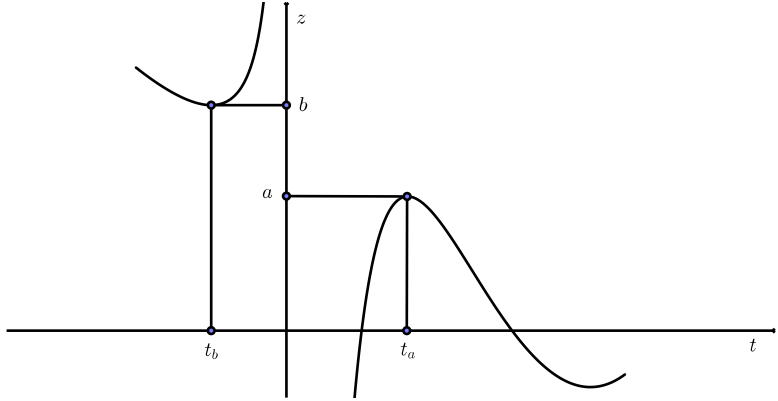}

\caption{\label{fig:zfunct}The function $-P(t)/t$ when $t\in\mathbb{R}$}
\end{figure}

The remainder of the paper is dedicated entirely to the proof of Theorem
\ref{maintheorem}, with the exception of the closing section, where
we state some open problems and conjectures which arose during our
investigations.

\section{The proof of Theorem \ref{maintheorem}}

Our approach to proving the main result is straightforward in that
we simply count the number of positive real zeros of the polynomials
$H_{m}(z)$, and show that for large $m$, this number is at least
the degree of $H_{m}(z)$. We start by giving an upper bound on the
degree of each $H_{m}(z)$. 
\begin{lem}
\label{lem:degreeH} Suppose that the sequence of polynomials ${\displaystyle {\left\{ H_{m}(z)\right\} _{m=0}^{\infty}}}$
is generated by a function $1/(P(t)+zt^{r})$, where $P\in\mathbb{R}[t]$
is a polynomial of degree $n$ with only positive (real) zeros, and
$r$ is a positive integer. For all $m\in\mathbb{N}$, the inequality
$\deg H_{m}(z)\leq\lfloor m/r\rfloor$ holds. 
\end{lem}
\begin{proof}
Rearranging (\ref{eq:genrel}) yields the equation 
\[
(P(t)+zt^{r})\sum_{m=0}^{\infty}H_{m}(z)t^{m}=1.
\]
By equating coefficients we see that the polynomial $H_{m}(z)$ satisfies
the recurrence 
\begin{equation}
(P(\Delta)+z\Delta^{r})H_{m}(z)=0,\qquad m\ge1\label{eq:recurrence}
\end{equation}
where the operator $\Delta$ is defined by $\Delta H_{m}:=H_{m-1}$,
and $\Delta H_{0}=0$. The claim follows from induction. 
\end{proof}
A key component of the proof of Theorem \ref{maintheorem} is the
connection between the zeros of $P(t)+zt^{r}$ (as a polynomial in
$t$) and the zeros of the generated sequence of polynomials $\{H_{m}(z)\}_{m=0}^{\infty}$.
The next segment of the paper starts the exploration of this connection.
More precisely, we proceed by demonstrating that given a $\theta\in(0,\pi/r)$,
one can find a unique $\tau\in\mathbb{R}^{+}$, such that $t=\tau e^{-i\theta}$
is a zero of $P(t)+zt^{r}$ for some real $z$. This way, we will
be able construct a function $z:(0,\pi/r)\to\mathbb{R}^{+}$, that
will play an instrumental role in the proof of Theorem \ref{maintheorem}.\\
 \indent In order to motivate some of the specifics of the ensuing
section, we offer the following development. Suppose $P(t)=\sum_{k=0}^{n}a_{k}t^{k}$
is a real polynomial with positive zeros $\tau_{1}\leq\tau_{2}\leq\cdots\leq\tau_{n}$.
If $z\in\mathbb{R}$, and $t=|t|e^{-i\theta}$, $\theta\in\mathbb{R}$,
is a zero of $P(t)+zt^{r}$, then so is $te^{2i\theta}$. For non-zero
values of $\theta$ we substitute $t$ and $te^{2i\theta}$ into the
equation $P(t)+zt^{r}=0$ to obtain 
\begin{equation}
\prod_{k=1}^{n}\frac{\tau_{k}-te^{2i\theta}}{\tau_{k}-t}=e^{2ir\theta}.\label{eq:proddistances}
\end{equation}
We define the angles $\theta_{k}$ implicitly by 
\begin{equation}
\frac{\tau_{k}-te^{2i\theta}}{\tau_{k}-t}=e^{2i\theta_{k}},\qquad k=1,2,\ldots,n.\label{eq:theta_kdef}
\end{equation}
Note that 
\begin{equation}
\sum_{k=1}^{n}\theta_{k}=r\theta+l\pi,\quad\text{for some}\quad0\leq l<n.\label{eq:sumthetak}
\end{equation}
Solving equation (\ref{eq:theta_kdef}) for $t$ we obtain that for
$k=1,2,\ldots,n$, 
\begin{equation}
t=\tau_{k}\frac{1-e^{2i\theta_{k}}}{e^{2i\theta}-e^{2i\theta_{k}}}=\tau_{k}e^{-i\theta}\frac{e^{-i\theta_{k}}-e^{i\theta_{k}}}{e^{i(\theta-\theta_{k})}-e^{-i(\theta-\theta_{k})}}=\tau_{k}\frac{\sin\theta_{k}}{\sin(\theta_{k}-\theta)}e^{-i\theta}.\label{eq:t0form}
\end{equation}
The quantity $\tau={\displaystyle {\tau_{k}\frac{\sin\theta_{k}}{\sin(\theta_{k}-\theta)}}}$
is the unique real number $\tau$ that we mentioned above - its existence
is established in Lemma \ref{lem:thetatuple}. Using the expression
for $t$ given in (\ref{eq:t0form}) we arrive at 
\begin{equation}
\tau_{k}-t=\tau_{k}\frac{-\cos\theta_{k}\sin\theta+i\sin\theta_{k}\sin\theta}{\sin(\theta_{k}-\theta)}=-\tau_{k}\frac{\sin\theta}{\sin(\theta_{k}-\theta)}e^{-i\theta_{k}}.\label{eq:taukt0diff}
\end{equation}
Finally, we conclude that if $t$ is a zero of $P(t)+zt^{r}$, then
for every $1\le k\le n$, we may write 
\begin{equation}
z=-\frac{|a_{n}|\prod_{k=1}^{n}(\tau_{k}-t)}{t^{r}}=|a_{n}|(-1)^{n-l-1}\frac{\sin^{n}\theta\sin^{r}(\theta_{k}-\theta)}{\tau_{k}^{r}\sin^{r}\theta_{k}}\prod_{j=1}^{n}\frac{\tau_{j}}{\sin(\theta_{j}-\theta)}.\label{eq:ztheta}
\end{equation}
Note in particular, that given a $\theta\in(0,\pi/r)$ and an $1\leq l<n$,
we may define a real $z$ as in (\ref{eq:ztheta}), so that $\tau e^{-i\theta}$
is a zero of $P(t)+zt^{r}$.\\
 \indent Motivational interlude aside, we now turn our attention
to the details of establishing the existence and the properties of
the function $z(\theta)$. 
\begin{lem}
\label{lem:zerosRdisct} Let $P(t)=\sum_{k=0}^{n}a_{k}t^{k}=|a_{n}|\prod_{k=1}^{n}(\tau_{k}-t)$
be a real polynomial of degree $n$, whose zeros $\tau_{1}\le\tau_{2}\le\cdots\le\tau_{n}$
are positive real. For any $\theta\in\mathbb{R}\setminus\{0\}$ and
$r\in\mathbb{N}$, the zeros of the polynomial 
\begin{equation}
S(\tau)=P(\tau e^{i\theta})-e^{2ir\theta}P(\tau e^{-i\theta})\label{eqn:Rtau}
\end{equation}
are real and distinct. 
\end{lem}
\begin{proof}
We first show that the zeros of $S(\tau)$ are real. To this end,
suppose that $S(\tau)=0$, $\tau\in\mathbb{C}$. Rearranging equation
(\ref{eqn:Rtau}) yields 
\begin{equation}
\prod_{k=1}^{n}\frac{\tau_{k}-\tau e^{i\theta}}{\tau_{k}-\tau e^{-i\theta}}=e^{2ir\theta}.\label{eqn:Rtauprodzero}
\end{equation}
If $\tau\notin\mathbb{R}$, then $\tau_{1},\ldots,\tau_{n}$ all lie
in the open half plane containing the positive real axis, with boundary
the perpendicular bisector of the line segment joining the points
$\tau e^{i\theta}$ and $\tau e^{-i\theta}$. Consequently, 
\[
\prod_{k=1}^{n}\left|\frac{\tau_{k}-\tau e^{i\theta}}{\tau_{k}-\tau e^{-i\theta}}\right|\ne1,
\]
a contradiction. \\
 Let $\gamma_{1},\ldots,\gamma_{n-1}$ denote the zeros of $P'(\tau)$
(see Figure \ref{fig:zerosR} for an illustration). If $\tau\in\mathbb{R}$
is such that $S(\tau)=S'(\tau)=0$, then along with equation (\ref{eqn:Rtau}),
we also have the following equation: 
\begin{equation}
\prod_{k=1}^{n-1}\frac{\gamma_{k}-\tau e^{i\theta}}{\gamma_{k}-\tau e^{-i\theta}}=e^{2i(r-1)\theta}.\label{eqn:Rprimetauprodzero}
\end{equation}
We define the angles $0<\theta_{k},\eta_{k}<\pi$ implicitly by 
\begin{eqnarray*}
\frac{\tau_{k}-\tau e^{i\theta}}{\tau_{k}-\tau e^{-i\theta}} & = & e^{2i\theta_{k}}\qquad(k=0,1,2,\ldots,n)\\
\frac{\gamma_{k}-\tau e^{i\theta}}{\gamma_{k}-\tau e^{-i\theta}} & = & e^{2i\eta_{k}}\qquad(k=0,1,2,\ldots,n-1).
\end{eqnarray*}

\begin{figure}[H]
\includegraphics[scale=0.4]{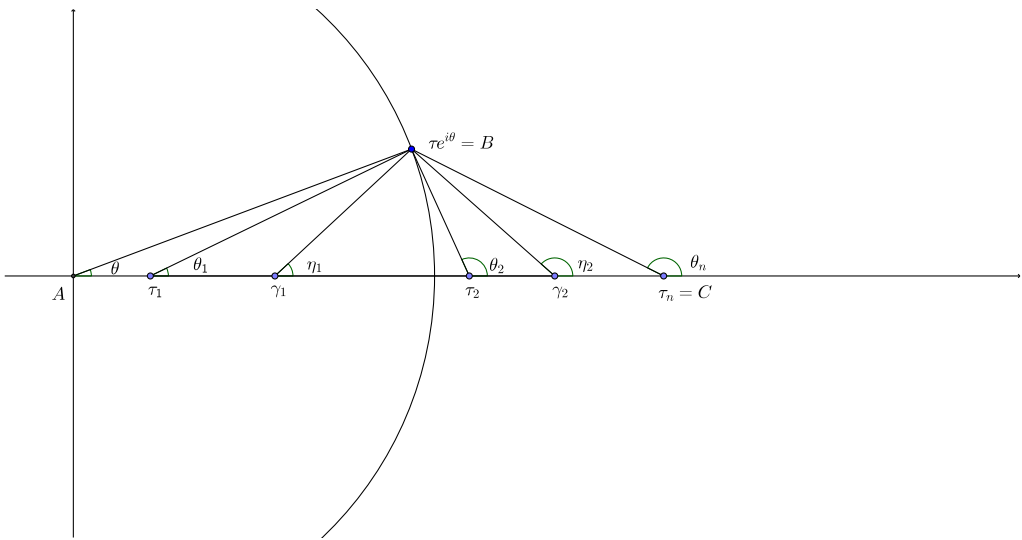}

\caption{\label{fig:zerosR}Zeros of $P(\tau)$ and $P'(\tau)$}
\end{figure}

With these definitions, equations (\ref{eqn:Rtauprodzero}) and (\ref{eqn:Rprimetauprodzero})
imply that 
\[
\sum_{k=1}^{n}\theta_{k}\equiv\sum_{k=1}^{n-1}\eta_{k}+\theta\mbox{ (mod }\pi).
\]
On the other hand, the interlacing of the zeros of $P$ and $P'$,
together with the fact that $\angle ABC<\pi$ imply that $\sum_{k=1}^{n-1}\eta_{k}+\theta<\sum_{k=1}^{n}\theta_{k}<\sum_{k=1}^{n-1}\eta_{k}+\theta+\pi$,
a contradiction. We conclude that the zeros of $S(\tau)$ are simple,
as well as real. 
\end{proof}
\begin{rem}
\label{rem:anglesumdecrease} We emphasize the following consequence
of Lemma \ref{lem:zerosRdisct}. Given $\theta\in(0,\pi/r)$, the
angle sum $\sum_{k=1}^{n}\theta_{k}$ is decreasing as $\tau$ increases.
Consequently, given the $n$-tuple $(\theta_{1},\theta_{2},\ldots,\theta_{n})$
corresponding to $\theta$ via equation (\ref{eq:theta_kdef}), the
equality in (\ref{eq:sumthetak}) is satisfied for a \foreignlanguage{english}{\textit{unique}
value $0\leq l<n$. }
\end{rem}
In our motivational interlude we indicated that an angle $\theta\in(0,\pi/r)$
would generate an $n$-tuple of angles $(\theta_{1},\theta_{2},\ldots,\theta_{n})$
with certain desirable properties. In order to make these properties
explicit, we need to make use of the complex version of the Implicit
Function Theorem. 
\begin{thm}[Theorem 2.1.2,\, p.24 \cite{hormander}]
\label{IFT} Let $f_{j}(w,z),$ $j=1,\ldots,m$, be analytic functions
of $(w,z)=(w_{1},\ldots,w_{m},z_{1},\ldots,z_{n})$ in a neighborhood
of a point $(w^{*},z^{*})$ in $\mathbb{C}^{m}\times\mathbb{C}^{n}$,
and assume that $f_{j}(w^{*},z^{*})=0$, $j=1,\ldots,m$, and that
\[
\det\left(\frac{\partial f_{j}}{\partial w_{k}}\right)_{j,k=1}^{m}\neq0\qquad\textrm{at}\quad(w^{*},z^{*}).
\]
Then the equations $f_{j}(w,z)=0$, $j=1,\ldots,m$ have a uniquely
determined analytic solution $w(z)$ in a neighborhood of $z^{*}$,
such that $w(z^{*})=w^{*}$. 
\end{thm}
We are now ready to state and prove the following lemma. 
\begin{lem}
\label{lem:thetatuple} Let $P(t)$ be a polynomial of degree $n$
whose zeros, $\tau_{1}\le\tau_{2}\le\cdots\le\tau_{n}$, are positive
and let $0\le l<n$ be a fixed integer. There is a unique $n$-tuple
of functions $\theta_{1}=\theta_{1}(\theta),\ldots,\theta_{n}=\theta_{n}(\theta)$
analytic on a neighborhood $\mathcal{U}$ of $(0,\pi/r)$ such that
for all $\theta\in(0,\pi/r)$ the following hold: 
\begin{itemize}
\item[(i)] $\theta_{k}(\theta)>\theta$, $1\le k\le n$, 
\item[(ii)] $\sum_{k=1}^{n}\theta_{k}(\theta)=r\theta+l\pi$, and 
\item[(iii)] 
\begin{equation}
\tau_{1}\frac{\sin\theta_{1}}{\sin(\theta_{1}-\theta)}=\tau_{2}\frac{\sin\theta_{2}}{\sin(\theta_{2}-\theta)}=\cdots=\tau_{n}\frac{\sin\theta_{n}}{\sin(\theta_{n}-\theta)}=:\tau(\theta).\label{eq:taudef}
\end{equation}
\end{itemize}
\end{lem}
\begin{proof}
Let $\theta^{*}\in(0,\pi/r)$ be given. By Lemma \ref{lem:zerosRdisct},
there is a unique solution $\tau^{*}$ to equation (\ref{eq:proddistances})
with the corresponding $n$-tuple $(\theta_{1}^{*},\ldots,\theta_{n}^{*})$
satisfying $\theta^{*}<\theta_{1}^{*}\leq\cdots\leq\theta_{n}^{*}$,
and $\sum_{k=1}^{n}\theta_{k}^{*}=r\theta^{*}+l\pi$. 
\begin{figure}[H]
\includegraphics[scale=0.4]{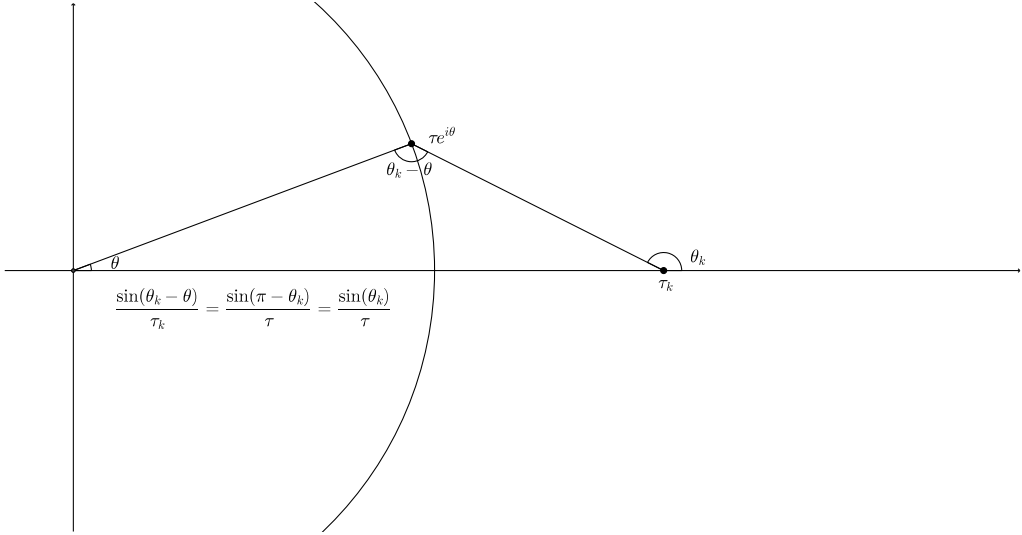} \caption{\label{fig:lawofsines}A representative triangle}
\end{figure}

The law of sines implies (see Figure \ref{fig:lawofsines}) that for
all $1\leq k\leq n$, 
\[
\tau^{*}=\tau_{k}\frac{\sin\theta_{k}^{*}}{\sin(\theta_{k}^{*}-\theta^{*})}.
\]
For $k=1,\ldots,n$ define $f_{k}:\mathbb{C}^{n+1}\times\mathbb{C}\to\mathbb{C}$
by 
\[
f_{k}(\theta_{1},\theta_{2}\ldots,\theta_{n},\tau,\theta)=\tau_{k}\frac{\sin\theta_{k}}{\sin(\theta_{k}-\theta)}-\tau
\]
and 
\[
f_{n+1}(\theta_{1},\theta_{2}\ldots,\theta_{n},\tau,\theta)=\sum_{k=1}^{n}\theta_{k}-(r\theta+l\pi).
\]
Note that 
\[
f_{k}(\theta_{1}^{*},\ldots,\theta_{n}^{*},\tau^{*},\theta^{*})=0,\qquad1\le k\le n+1,
\]
and that there exists a neighborhood $\mathcal{W}$ of $(\theta_{1}^{*},\ldots,\theta_{n}^{*},\tau^{*},\theta^{*})\in\mathbb{C}^{n+1}\times\mathbb{C}$
where each $f_{k}(\theta_{1},\ldots,\theta_{n},\tau,\theta)$ is analytic.
We calculate 
\[
\frac{\partial f_{k}}{\partial\theta_{k}}\Bigg|_{(\theta_{1}^{*},\theta_{2}^{*},\ldots,\theta_{n}^{*},\tau^{*},\theta^{*})}=\frac{-\tau_{k}\sin\theta^{*}}{\sin^{2}(\theta_{k}^{*}-\theta^{*})}=:c_{k}<0,\qquad1\leq k\leq n,
\]
and hence the Jacobian matrix at $(\theta_{1}^{*},\theta_{2}^{*},\ldots,\theta_{n}^{*},\tau^{*},\theta^{*})$
is 
\[
\left[\begin{array}{ccccc}
c_{1} & 0 & \cdots & 0 & -1\\
0 & c_{2} & \cdots & 0 & -1\\
\vdots & \vdots & \ddots & \vdots & \vdots\\
0 & 0 & \cdots & c_{n} & -1\\
1 & 1 & \cdots & 1 & 0
\end{array}\right].
\]
By expanding along the first row, we find the determinant of this
matrix to be 
\[
c_{2}\cdots c_{n}+c_{1}c_{3}\cdots c_{n}+c_{1}c_{2}c_{4}\cdots c_{n}+\cdots+c_{1}c_{2}\cdots c_{n-1}\ne0,
\]
since each summand carries the same sign. We may therefore invoke
Theorem \ref{IFT} to conclude that there is a unique analytic function
$w(\theta)=(\theta_{1}(\theta),\theta_{2}(\theta),\ldots,\theta_{n}(\theta))$
on a neighborhood $\mathcal{U}_{\theta^{*}}$ of $\theta^{*}$, such
that 
\[
w(\theta^{*})=(\theta_{1}(\theta^{*}),\theta_{2}(\theta^{*}),\ldots,\theta_{n}(\theta^{*}))=(\theta_{1}^{*},\theta_{2}^{*},\ldots,\theta_{n}^{*}).
\]
By analytic continuation we conclude that there is a neighborhood
of $\mathcal{U}$ of $(0,\pi/r)$, and an analytic function $w$ on
$\mathcal{U}$ such that $w(\theta^{*})=(\theta_{1}^{*},\ldots,\theta_{n}^{*})$
for any $\theta^{*}\in(0,\pi/r)$. Naming the components of $w$ as
$\theta_{k}(\theta)$, $1\leq k\leq n$, finishes the proof. 
\end{proof}
\begin{lem}
\label{lem:zfunctheta} Let $0\le l<n$, and set $t_{0}=\tau e^{-i\theta}$,
where $\tau=\tau(\theta)$ is as in Lemma \ref{lem:thetatuple} (iii).
The real-valued function 
\[
z(\theta)=-\frac{P(t_{0})}{t_{0}^{r}}
\]
is strictly monotone on $(0,\pi/r)$. 
\end{lem}
\begin{proof}
Note that since $t_{0}\neq0$, and $\theta-\theta_{k}\neq0$ for any
$1\leq k\leq n$, the function $z(\theta)$ is differentiable on $(0,\pi/r)$.
We take the logarithmic derivatives of both sides of $z=-P(t_{0})/t_{0}^{r}$
to obtain 
\begin{equation}
\frac{dz}{z}=-\sum_{k=1}^{n}\frac{dt_{0}}{\tau_{k}-t_{0}}-r\frac{dt_{0}}{t_{0}}.\label{eq:dzt}
\end{equation}
Similarly, the logarithmic derivatives of both sides of the equation
\[
\prod_{k=1}^{n}\tau_{k}-t_{0}e^{2i\theta}=e^{2ir\theta}\prod_{k=1}^{n}\tau_{k}-t_{0}
\]
give 
\[
\sum_{k=1}^{n}\frac{-e^{2i\theta}dt_{0}-2it_{0}e^{2i\theta}d\theta}{\tau_{k}-t_{0}e^{2i\theta}}=2ird\theta+\sum_{k=1}^{n}\frac{-dt_{0}}{\tau_{k}-t_{0}},
\]
or equivalently 
\begin{equation}
\sum_{k=1}^{n}\frac{-\sin\theta\tau_{k}e^{i\theta}}{(\tau_{k}-t_{0})(\tau_{k}-t_{0}e^{2i\theta})}dt_{0}=rd\theta+\sum_{k=1}^{n}\frac{t_{0}e^{2i\theta}}{\tau_{k}-t_{0}e^{2i\theta}}d\theta.\label{eq:dthetat}
\end{equation}
Thus (\ref{eq:dzt}) and (\ref{eq:dthetat}) yield 
\[
\frac{dz}{zd\theta}\sum_{k=1}^{n}\frac{\sin\theta\tau_{k}e^{i\theta}}{(\tau_{k}-t_{0})(\tau_{k}-t_{0}e^{2i\theta})}=\left(\sum_{k=1}^{n}\frac{1}{\tau_{k}-t_{0}}+\frac{r}{t_{0}}\right)\left(\sum_{k=1}^{n}\frac{t_{0}e^{2i\theta}}{\tau_{k}-t_{0}e^{2i\theta}}+r\right).
\]
Multiplying both sides by $t_{0}$ results in 
\begin{equation}
\frac{dz}{zd\theta}\sum_{k=1}^{n}\frac{\tau_{k}\tau(\theta)\sin\theta}{|\tau_{k}-\tau e^{-i\theta}|^{2}}=\left|\sum_{k=1}^{n}\frac{\tau e^{-i\theta}}{\tau_{k}-\tau e^{-i\theta}}+r\right|^{2},
\end{equation}
from which monotonicity follows\footnote{whether $z$ is increasing or decreasing depends on the parity of
$n-l-1$, as in (\ref{eq:ztheta})}. To establish strict monotonicity we note that ${\displaystyle {-\pi<\text{Arg}\frac{\tau e^{-i\theta}}{\tau_{k}-\tau e^{-i\theta}}<0}}$
for all $k$ and $\theta\in(0,\pi/r)$. Consequently, 
\[
-\pi<\text{Arg}\left(\sum_{k=1}^{n}\frac{\tau e^{-i\theta}}{\tau_{k}-\tau e^{-i\theta}}\right)<0,\qquad(\theta\in(0,\pi/r))
\]
and we conclude that 
\[
\sum_{k=1}^{n}\frac{\tau e^{-i\theta}}{\tau_{k}-\tau e^{-i\theta}}+r\neq0.\qquad(\theta\in(0,\pi/r))
\]
The proof is complete. 
\end{proof}
\begin{rem}
Note that Lemmas \ref{lem:thetatuple} and \ref{lem:zfunctheta} hold
for any fixed $0\le l<n$. In particular, for each $0\le l<n$, we
have a tuple of $n$ continuous real-valued functions $\theta_{k}(\theta;l)$,
$1\le k\le n$, on $\theta\in(0,\pi/r)$ where the function $z(\theta;l)$
given in \eqref{eq:ztheta} is monotone. For simplicity in notation,
from this point of the paper, whenever $l=n-1$, we will suppress
the parameter $l$ and simply write $\theta_{k}(\theta):=\theta_{k}(\theta;n-1)$,
$1\le k\le n$, $\tau(\theta):=\tau(\theta;n-1)$, and $z(\theta):=z(\theta;n-1)$.
If $l\ne n-1$, we will write these functions in their full notations
$\theta_{k}(\theta;l)$, $\tau(\theta;l)$, and $z(\theta;l)$.

An argument completely analogous to the proof of Lemma \ref{lem:thetatuple}
gives the next auxiliary result, which we will make use of when establishing
the last important property of the function $z(\theta)$. 
\end{rem}
\begin{lem}
\label{lem:Athetapositive} If $l=n-1$, the function $\tau(\theta)$
defined in (\ref{eq:taudef}) is strictly monotone decreasing on $(0,\pi/r)$. 
\end{lem}
\begin{proof}
With $t_{0}=\tau e^{-i\theta}$ and $dt_{0}=e^{-i\theta}d\tau-i\tau e^{-i\theta}d\theta$,
\eqref{eq:dthetat} gives 
\begin{eqnarray*}
\sum_{k=1}^{n}\frac{-\sin\theta\tau_{k}}{(\tau_{k}-\tau e^{-i\theta})(\tau_{k}-\tau e^{i\theta})}d\tau & = & rd\theta+\sum_{k=1}^{n}\left(\frac{\tau e^{i\theta}}{\tau_{k}-\tau e^{i\theta}}-\frac{i\sin\theta\tau\tau_{k}}{(\tau_{k}-\tau e^{-i\theta})(\tau_{k}-\tau e^{i\theta})}\right)d\theta\\
 & = & rd\theta-\sum_{k=1}^{n}\frac{\tau^{2}-\tau\tau_{k}\cos\theta}{(\tau_{k}-\tau e^{i\theta})(\tau_{k}-\tau e^{-i\theta})}d\theta.
\end{eqnarray*}
Using \eqref{eq:t0form} and \eqref{eq:taukt0diff} we may rewrite
the above equation as 
\begin{equation}
-\frac{d\tau}{\sin\theta d\theta}\sum_{k=1}^{n}\frac{\sin^{2}(\theta_{k}-\theta)}{\tau_{k}}=r-\sum_{k=1}^{n}\frac{\sin\theta_{k}}{\sin\theta}\cos(\theta_{k}-\theta).\label{eq:dtaudtheta}
\end{equation}
The claim will follow as soon as we establish that the right hand
side of (\ref{eq:dtaudtheta}) 
\begin{equation}
A(\theta):=r-\sum_{k=1}^{n}\frac{\sin\theta_{k}}{\sin\theta}\cos(\theta_{k}-\theta),\label{eqn:athetadef}
\end{equation}
is positive. We first consider the case when $r\ge n/2$. In this
case, we see that 
\begin{eqnarray*}
A(\theta)\sin\theta & = & \left(r-\frac{n}{2}\right)\sin\theta-\frac{1}{2}\sum_{k=1}^{n}\sin(2\theta_{k}-\theta).
\end{eqnarray*}
If $\pi\le2\theta_{k}-\theta<2\pi$, $1\le k\le n$, then the claim
$A(\theta)\ge0$ is trivial since all the terms are nonnegative. Furthermore
$A(\theta)=0$ only when $r=n/2$ and $2\theta_{k}-\theta=\pi$, $1\le k\le n$.
The sum all terms $1\le k\le n$ in the later equality implies that
$n=2$. On the other hand, if $2\theta_{1}-\theta<\pi$ then the identity
$\sum_{k=1}^{n}(2\theta_{k}-\theta)=2(n-1)\pi+(2r-n)\theta$ implies
that $0<2\theta_{1}-\theta<\pi$ and $\pi<2\theta_{k}-\theta<2\pi$,
$\forall k\ge2$. Let $\eta_{1}=\theta_{1}$ and $\eta_{k}=\pi-\theta_{k}$,
$k\ge2$, where $\eta_{1}-\sum_{k=2}^{n}\eta_{k}=r\theta$. In terms
of these new variables, we have 
\begin{equation}
A(\theta)\sin\theta=\left(r-\frac{n}{2}\right)\sin\theta-\frac{1}{2}\sin(2\eta_{1}-\theta)+\frac{1}{2}\sum_{k=2}^{n}\sin(2\eta_{k}+\theta)\label{eq:firstsumsine}
\end{equation}
where $2\eta_{1}-\theta=(2r-n)\theta+\sum_{k=2}^{n}(2\eta_{k}+\theta)$.
Furthermore, we have $0<2\eta_{1}-\theta<\pi$ and $0<2\eta_{k}+\theta<\pi$,
$\forall k\ge2$. We note that for any two angles $0<\alpha,\beta<\pi$
the following inequality holds 
\begin{equation}
\sin(\alpha+\beta)=\sin\alpha\cos\beta+\cos\alpha\sin\beta<\sin\alpha+\sin\beta.\label{eq:sineineq}
\end{equation}
Thus 
\[
A(\theta)\sin\theta>\left(r-\frac{n}{2}\right)\sin\theta-\frac{1}{2}\sin(2r\theta-n\theta)\ge0.
\]

We next consider when $n>2r$. In this case we write $A(\theta)$
as 
\[
A(\theta)\sin\theta=-\frac{1}{2}\sum_{k=1}^{2r}\sin(2\theta_{k}-\theta)-\sum_{k=2r+1}^{n}\sin\theta_{k}\cos(\theta_{k}-\theta).
\]
For the same reason, it suffices to consider when $0<2\theta_{1}-\theta<\pi$
and $\pi<2\theta_{k}-\theta<2\pi$, $\forall k\ge2$. In this case
\[
A(\theta)\sin\theta=-\frac{1}{2}\sin(2\eta_{1}-\theta)+\frac{1}{2}\sum_{k=2}^{2r}\sin(2\eta_{k}+\theta)+\sum_{k=2r+1}^{n}\sin\eta_{k}\cos(\eta_{k}+\theta),
\]
where $2\eta_{1}-\theta=\sum_{k=2}^{2r}(2\eta_{k}+\theta)+\sum_{k=2r+1}^{n}2\eta_{k}$.
We can write $\sin(2\eta_{1}-\theta)$ as 
\begin{equation}
\sin(2\eta_{n})\cos\left(\sum_{k=2}^{2r}(2\eta_{k}+\theta)+\sum_{k=2r+1}^{n-1}2\eta_{k}\right)+\cos(2\eta_{n})\sin\left(\sum_{k=2}^{2r}(2\eta_{k}+\theta)+\sum_{k=2r+1}^{n-1}2\eta_{k}\right)\label{eq:secondsumsine}
\end{equation}
which is less than 
\[
2\sin\eta_{n}\cos(\eta_{n}+\theta)+\sin\left(\sum_{k=2}^{2r}(2\eta_{k}+\theta)+\sum_{k=2r+1}^{n-1}2\eta_{k}\right)
\]
since 
\[
\eta_{n}+\theta\le\sum_{k=2}^{2r}(2\eta_{k}+\theta)+\sum_{k=2r+1}^{n-1}2\eta_{k}<\pi.
\]
We continue this process and obtain the inequality 
\begin{eqnarray}
\sin(2\eta_{1}-\theta) & \le & 2\sum_{k=2r+1}^{n}\sin\eta_{k}\cos(\eta_{k}+\theta)+\sin\left(\sum_{k=2}^{2r}(2\eta_{k}+\theta)\right)\label{eq:thirdsumsine}\\
 & \le & 2\sum_{k=2r+1}^{n}\sin\eta_{k}\cos(\eta_{k}+\theta)+\sum_{k=2}^{2r}\sin(2\eta_{k}+\theta)\nonumber 
\end{eqnarray}
where the last inequality comes from \eqref{eq:sineineq}. We conclude
that $A(\theta)>0$ if $(n,r)\neq(2,1)$, and the proof of the lemma
is complete. 
\end{proof}
\begin{rem}
From $\sum_{k=1}^{n}\theta_{k}=r\theta+(n-1)\pi$, not all $\theta_{k}$,
$1\le k\le n$, approach $0$ when $\theta\rightarrow0$. Thus the
previous lemma and \eqref{eq:taudef} imply that $\tau(\theta)>0$
is bounded on $(0,\pi/r)$. 
\end{rem}
We wrap up our discussion of the function $z(\theta)$ with the following
result. 
\begin{lem}
\label{lem:zonto} Let $l=n-1$, let $t_{0},\tau$ and $z(\theta)$
be as in the statement of Lemma \ref{lem:zfunctheta}, and let $t_{a},t_{b},a$
and $b$ be defined as in the statement of Theorem \ref{maintheorem}.
The function $z(\theta)$ maps the interval $(0,\pi/r)$ onto the
interval $(a,b)$. 
\end{lem}
\begin{proof}
Lemma \ref{lem:zfunctheta} implies that with the choice $l=n-1$,
$z(\theta)$ is a continuous, monotone increasing function on $(0,\pi/r)$.
Thus, it suffices to show $a':=\lim_{\theta\rightarrow0}z(\theta)=a$
and $b':=\lim_{\theta\rightarrow\pi/r}z(\theta)=b$. Since $\tau(\theta)e^{\pm i\theta}$
are two zeros of $P(t)+zt^{r}$ and $\tau(\theta)$ is a monotone
bounded function, the function $\tau(\theta)$ has to converge to
a real multiple zero of $P(t)+a't^{r}$when $\theta\rightarrow0$
or a real multiple zero of $P(t)+b't^{r}$ when $\theta\rightarrow\pi/r$.
Here, we use the convention that if $b'=\infty$ and $r>1$ then $0$
is a multiple zero of $P(t)+b't^{r}$. By looking at the derivatives
of $P(t)+a't^{r}$ and $P(t)+b't^{r}$, we see that the limits of
$\tau(\theta)$ as $\theta\rightarrow0$ and as $\theta\rightarrow\pi/r$
must be zeros of $t^{r-1}(rP(t)-P'(t))$ where the factor $t^{r-1}$
reflects the possible multiple zero at $0$ when $r>1$. The equation
$\sum_{k=1}^{n}\theta_{k}(\theta)=(n-1)\pi+r\theta$ implies that
$\tau_{1}\le t_{a'}\le\tau_{2}$ where $t_{a'}:=\lim_{\theta\rightarrow0}\tau(\theta)$.
Thus by the interlacing of the zeros of $P(t)$ and $P'(t)$ and the
definition of $t_{a}$ as the smallest positive real zero of $t^{r-1}(rP(t)-P'(t))$,
we must have $t_{a'}=t_{a}$ and consequently $a'=a$. Similarly,
we have $t_{b'}=0$ if $r>1$ and $t_{b'}<0$ if $r=1$ where $t_{b'}:=\lim_{\theta\rightarrow\pi/r}\tau(\theta)$.
By the definition of $t_{b}$ we have $t_{b'}=t_{b}$ and $b=b'$. 
\end{proof}
We now turn our attention to the second key element of the proof of
Theorem \ref{maintheorem}, namely the construction of the functions
$H(\theta;m)$ on $(0,\pi/r)$ with the property that $H(\theta;m)=0$
if and only if $H_{m}(z(\theta))=0$.

Recall that for each $0<\theta<\pi/r$ and $0\le l<n$, we set $t_{0}=\tau(\theta;l)e^{-i\theta}$.
Let $t_{1}=t_{0}e^{2i\theta}$ be its conjugate, and note that they
are both zeros of the polynomial $P(t)+z(\theta;l)t^{r}$. Let $t_{2},t_{3},\ldots,t_{\max\{n,r\}-1}$
denote the remaining zeros of $P(t)+z(\theta;l)t^{r}$, and let $q_{k}=t_{k}/t_{0}$,
$0\le k<\max\{n,r\}$. With this notation we may rearrange the equation
$P(t_{k})+z(\theta;l)t_{k}^{r}=0$ as 
\[
-z=\frac{\prod_{j=1}^{n}(\tau_{j}-t_{k})}{t_{k}^{r}}=\frac{\prod_{j=1}^{n}(\tau_{j}t_{0}^{-1}-q_{k})}{q_{k}^{r}}t_{0}^{n-r}.
\]
Changing of variables $\zeta_{k}:=q_{k}e^{-i\theta}$, $k=0,1,\ldots,\max\{n,r\}$,
and using (\ref{eq:t0form}) we arrive at the equivalent equation
\[
-z(t_{0}e^{i\theta})^{r-n}\zeta_{k}^{r}=\prod_{j=1}^{n}\left(\frac{\sin(\theta_{j}-\theta)}{\sin\theta_{j}}-\zeta_{k}\right),
\]
which, when combined with (\ref{eq:ztheta}), leads to 
\[
\prod_{j=1}^{n}\left(\frac{\sin(\theta_{j}-\theta)}{\sin\theta_{j}}-\zeta_{k}\right)+\zeta_{k}^{r}\prod_{j=1}^{n}\frac{\sin\theta}{\sin\theta_{j}}=0.
\]
We deduce that for any $0\le k<\max\left\{ n,r\right\} $, $t_{k}$
is a zero of $P(t)+z(\theta;l)t^{r}$ if and only if $\zeta_{k}$
is a zero of the polynomial 
\[
Q(\zeta):=\prod_{j=1}^{n}\left(\frac{\sin(\theta_{j}-\theta)}{\sin\theta}-\zeta\frac{\sin\theta_{j}}{\sin\theta}\right)+\zeta^{r}.
\]
Note that $\zeta_{0}=e^{-i\theta}$ and $\zeta_{1}=e^{i\theta}$ are
both zeros of $Q(\zeta)$.

Let $c$ be the leading coefficient in $t$ of the polynomial $P(t)+zt^{r}$.
If the all the zeros $t_{k}$, $0\le k<\max\left\{ n,r\right\} $,
of the denominator $P(t)+zt^{r}$ are distinct, then partial fraction
decomposition 
\begin{eqnarray*}
\sum_{m=0}^{\infty}H_{m}(z)t^{m} & = & \frac{1}{c}\prod_{k=1}^{\max\left\{ n,r\right\} -1}\frac{1}{t-t_{k}}\\
 & = & \frac{1}{c}\sum_{k=0}^{\max\left\{ n,r\right\} -1}\frac{1}{t-t_{k}}\prod_{l\ne k}\frac{1}{t_{k}-t_{l}}\\
 & = & -\frac{1}{c}\sum_{m=0}^{\infty}\left(\sum_{k=0}^{\max\left\{ n,r\right\} -1}\frac{1}{t_{k}^{m+1}}\prod_{l\ne k}\frac{1}{t_{k}-t_{l}}\right)t^{m}.
\end{eqnarray*}
Thus, $H_{m}(z)=0$ if and only if 
\[
\sum_{k=0}^{\max\left\{ n,r\right\} -1}\frac{1}{t_{k}^{m+1}}\prod_{l\ne k}\frac{1}{t_{k}-t_{l}}=0.
\]
Multiplying the above equation by $e^{(m+n)i\theta}t_{0}^{m+\max\left\{ n,r\right\} }$,
we see that $z$ is a zero of $H_{m}(z)$ if and only if $\theta$
is a zero of 
\begin{eqnarray}
H(\theta) & := & \sum_{k=0}^{\max\left\{ n,r\right\} -1}\frac{1}{\zeta_{k}^{m+1}}\prod_{l\ne k}\frac{1}{\zeta_{k}-\zeta_{l}}\nonumber \\
 & = & \sum_{k=0}^{\max\{n,r\}-1}\frac{1}{\zeta_{k}^{m+1}Q'(\zeta_{k})}.\label{eq:Htheta}
\end{eqnarray}
By symmetric reduction and Lemma \ref{lem:thetatuple}, $H(\theta)$
is a real-valued function of $\theta$ on $(0,\pi/r)$. Furthermore,
$H(\theta)$ is analytic in an open neighborhood of $(0,\pi/r)$ except
at some values of $\theta$ where $\zeta_{k}$ (or $t_{k}$), $0\le k<\max\left\{ n,r\right\} $,
are not all distinct. However from Lemma \ref{lem:thetatuple} and
(\ref{eq:ztheta}), $z$ is an analytic function of $\theta$ on a
neighborhood of $(0,\pi/r)$. Thus $H(\theta)$ has an analytic continuation
$-c(t_{0}e^{i\theta})^{m+\max\left\{ n,r\right\} }H_{m}(z(\theta))$
on an neighborhood of $(0,\pi/r)$ and consequently all the discontinuities
of the real-valued function $H(\theta)$ on $(0,\pi/r)$ are removable.
From now on, we treat $H(\theta)$ as a real-valued continuous function
on $(0,\pi/r)$.

If $\zeta_{k}$ is a zero of $Q(\zeta)$, then 
\begin{eqnarray}
Q'(\zeta_{k}) & = & \prod_{j=1}^{n}\left(\frac{\sin(\theta_{j}-\theta)}{\sin\theta}-\zeta\frac{\sin\theta_{j}}{\sin\theta}\right)\sum_{j=1}^{n}\frac{-\sin\theta_{j}}{\sin(\theta_{j}-\theta)-\zeta_{k}\sin\theta_{j}}+r\zeta_{k}^{r-1}\nonumber \\
 & = & \sum_{j=1}^{n}\frac{\zeta_{k}^{r}\sin\theta_{j}}{\sin(\theta_{j}-\theta)-\zeta_{k}\sin\theta_{j}}+r\zeta_{k}^{r-1}.\label{eq:Qprimeatzero}
\end{eqnarray}
In the instance $\zeta_{1}=e^{i\theta},$ we apply the identity 
\[
\sin(\theta_{j}-\theta)-e^{i\theta}\sin\theta_{j}=-\cos\theta_{j}\sin\theta-i\sin\theta\sin\theta_{j}=-\sin\theta e^{i\theta_{j}}
\]
and obtain the equation 
\[
Q'(e^{i\theta})=e^{(r-1)i\theta}\left(-\sum_{j=1}^{n}\frac{\sin\theta_{j}}{\sin\theta}e^{i(\theta-\theta_{j})}+r\right).
\]
The first two terms of (\ref{eq:Htheta}) are complex conjugates,
whose sum is 
\begin{equation}
2\frac{\Re\left(e^{i(m+1)\theta}Q'(e^{i\theta})\right)}{|Q'(e^{i\theta})|^{2}}=\frac{2}{|Q'(e^{i\theta})|^{2}}.\left(A(\theta)\cos(m+r)\theta-B(\theta)\sin(m+r)\theta\right)\label{eq:mainterm}
\end{equation}
where $A(\theta)$ is as in (\ref{eqn:athetadef}), and ${\displaystyle {B(\theta)=\sum_{j=1}^{n}\frac{\sin\theta_{j}}{\sin\theta}\sin(\theta_{j}-\theta)}}$.

The next lemma describes the behavior of $A(\theta)$ near the end-points
of the interval $(0,\pi/r)$. 
\begin{lem}
\label{lem:AthetaAsymp}Let $\rho$ be the multiplicty of the smallest
zero of $P(t)$. When $\rho=1$ and $\theta$ sufficiently small,
$A(\theta)>\theta^{4}$. Moreover, when $r=1$ and $\theta$ sufficiently
close to $\pi$, $A(\theta)>(\pi-\theta)^{4}$. 
\end{lem}
\begin{proof}
We follow the proof of Lemma \ref{lem:Athetapositive} and replace
the identity \eqref{eq:sineineq} by the identity 
\begin{eqnarray}
\sin\alpha+\sin\beta-\sin(\alpha+\beta) & = & \sin\alpha(1-\cos\beta)+\sin\beta(1-\cos\alpha)\nonumber \\
 & > & \frac{\alpha\beta^{2}}{4}+\frac{\beta\alpha^{2}}{4}\label{eq:sinedifference}
\end{eqnarray}
for all $\alpha,\beta>0$ sufficiently small.

In the case $\theta$ is sufficiently small and $\rho=1$, we note
that $\eta_{k}$, $1\le k\le n$, are close to $0$. Moreover, the
law of sines (see Figure \ref{fig:lawofsines}) implies that $\sin\eta_{k}/\sin\theta$,
or equivalently $\eta_{k}/\theta$, is approaching $t_{a}/|t_{a}-\tau_{k}|$
when $\theta\rightarrow0$. The lemma follows from the proof of Lemma
\ref{lem:Athetapositive} after we apply \eqref{eq:sinedifference}
in \eqref{eq:firstsumsine}, \eqref{eq:secondsumsine} or \eqref{eq:thirdsumsine}.

Similarly, when $\theta$ is close to $\pi$ and $r=1$, $\sin\eta_{k}/\sin(\pi-\theta)$,
or equivalently $\eta_{k}/(\pi-\theta)$, is approaching $|t_{b}|/|t_{b}-\tau_{k}|$
for $2\le k\le n$. We apply \eqref{eq:sinedifference} in \eqref{eq:secondsumsine}
or \eqref{eq:thirdsumsine} and reach the corresponding conclusion. 
\end{proof}
Lemmas \ref{lem:Athetapositive}, \ref{lem:zonto}, and \ref{lem:AthetaAsymp}
establish some key properties of the functions $\tau(\theta),z(\theta)$
and $A(\theta)$. Equipped with these results, we can now prove a
key proposition about the distribution of the zeros of the denominator
$P(t)+zt^{r}$ as a polynomial in $t$. 
\begin{prop}
\label{prop:zerodenom}Let $\theta\in(0,\pi/r)$ be a fixed angle
with $z:=z(\theta)$ and $\tau:=\tau(\theta)$. The only zeros in
$t$ of $P(t)+zt^{r}$ in the closed disk $C_{1}$ centered at the
origin with radius $\tau$ are $t_{0,1}=\tau e^{\pm i\theta}$. 
\end{prop}
\begin{proof}
Without loss of generality, we may assume the zeros under consideration
lie in the upper half plane. If $t$ is a zero of $P(t)+zt^{r}$,
then 
\begin{equation}
\frac{\prod_{k=1}^{n}(\tau_{k}-t)}{t^{r}}=\frac{\prod_{k=1}^{n}(\tau_{k}-\tau e^{i\theta})}{(\tau e^{i\theta})^{r}}.\label{eqn:circles}
\end{equation}
Let $C_{1}$ be as in the statement of the theorem, let and $C_{2}$
be the closed disk centered at $\tau_{1}$ with radius $|\tau e^{i\theta}-\tau_{1}|$
(see Figures \ref{fig:ZerosOutsideDiskbigrho} and \ref{fig:ZerosOutsideDiskrhoeq1}
depending on the multiplicity of the smallest zero of $P(t)$). Assume,
by way of contradiction, that $t^{*}$ is a zero of $P(t)+zt^{r}$
which lies inside $C_{1}$. Equation (\ref{eqn:circles}) implies
that $t^{*}$ must then lie in the closed region $C_{1}\cap C_{2}$.
\begin{figure}[H]
\includegraphics[scale=0.4]{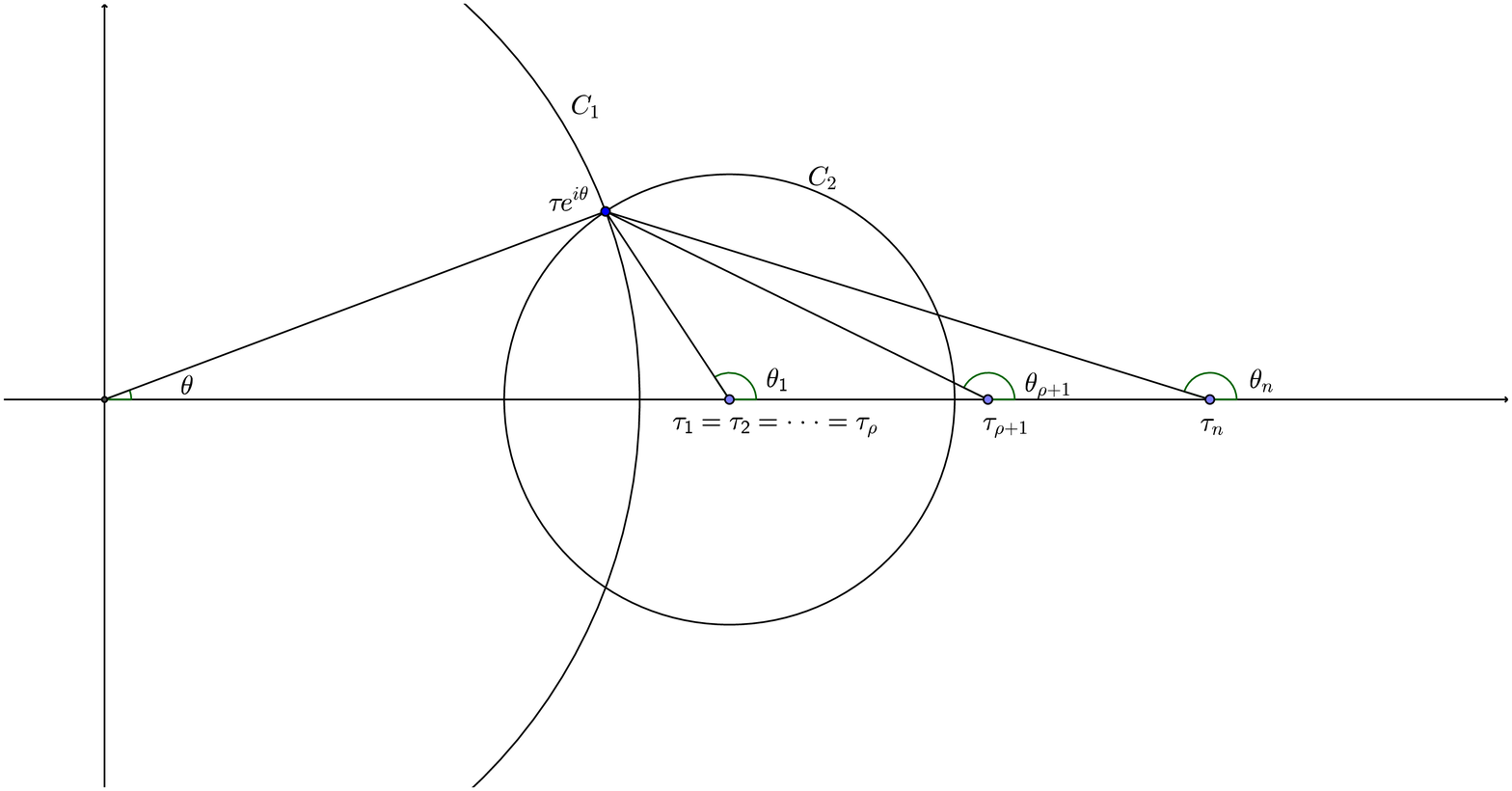}

\caption{\label{fig:ZerosOutsideDiskbigrho}Zeros of $P(t)+zt^{r}$ when $\rho>1$}
\end{figure}

\begin{figure}[H]
\includegraphics[scale=0.4]{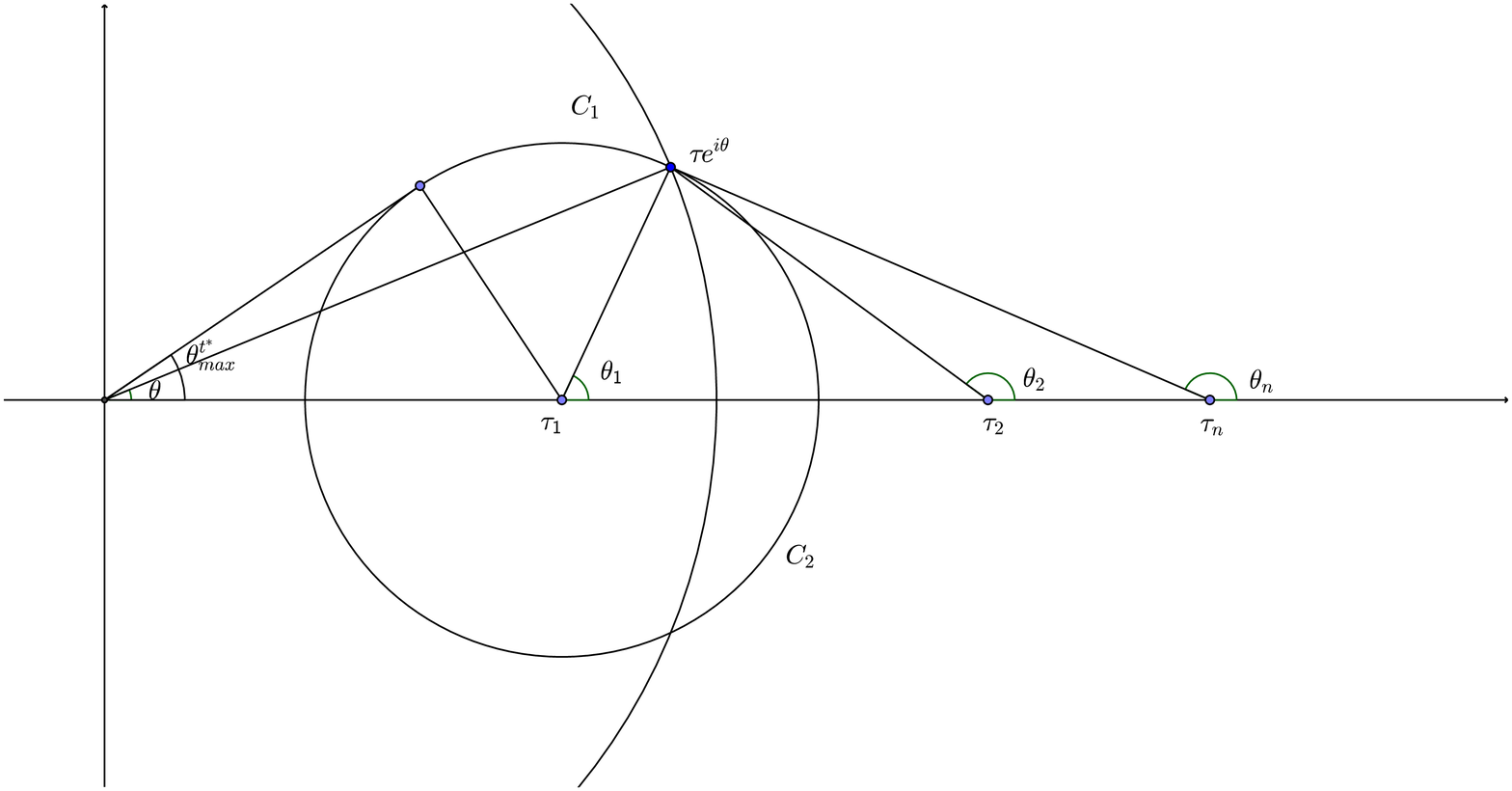}

\caption{\label{fig:ZerosOutsideDiskrhoeq1}Zeros of $P(t)+zt^{r}$ when $\rho=1$}
\end{figure}


We first argue that $t^{*}\notin\mathbb{R}$. If $\theta\in(0,\pi/r)$
and $P(t^{*})+z(t^{*})^{r}=0$, Lemma \ref{lem:zonto} implies that
$z\in(a,b)$, and hence $a<-P(t^{*})/(t^{*})^{r}<b$. Since the derivative
${\displaystyle {\frac{d}{dt}(-P(t)/t^{r})}}$ does not vanish on
$(0,t_{a})$ and $P(0)>0$, we conclude that $-P(t)/t^{r}$ is strictly
increasing on $(0,t_{a})$ with $-P(t_{a})/t_{a}^{r}=a$. It follows
that $t^{*}\notin[0,t_{a}]$. On the other hand, $\tau_{1}<t_{a}$,
and By Lemma \ref{lem:Athetapositive}, $\tau<t_{a}$ as well. We
deduce that regardless of the value of $\rho$, $C_{1}\cap C_{2}\cap\mathbb{R}\subsetneq[0,t_{a}]$,
and $P(t)+zt^{r}$ has no real non-negative zeros. If $r=1$, then
$-P(t)/t\geq b$ on $(-\infty,0)$ by Lemma \ref{lem:existenceinterval},
and hence $t^{*}\notin(-\infty,0)$. It now follows that $t^{*}\notin C_{1}\cap C_{2}\cap\mathbb{R}$
for any $r\geq1$.

We next claim that the argument $\theta^{t^{*}}$ of $t^{*}$ lies
in the interval $(0,\pi/r)$. This claim is trivial when $r=1$, since
$t^{*}\notin\mathbb{R}$. When $r>1$, it suffices to consider $\theta_{1}-\theta<\pi/2$,
for if $\theta_{1}-\theta\geq\pi/2$, then $t^{*}\in C_{1}\cap C_{2}$
forces $\theta^{t}\le\theta<\pi/r$ (see Figure \ref{fig:ZerosOutsideDiskbigrho}).
In addition, the equality $\sum_{k=1}^{n}\theta_{k}=r\theta+(n-1)\pi$
implies that $\theta_{1}>r\theta$, and consequently we may restrict
our attention to angles which satisfy $\theta<\pi/(2r-2).$ Since
$t^{*}\in C_{1}\cap C_{2}$, the largest possibe value of $\theta^{t^{*}}$,
denoted by $\theta_{\max}^{t^{*}}$, satisfies 
\[
\sin\theta_{\max}^{t^{*}}=\frac{\sin\theta}{\sin(\theta_{1}-\theta)}<\frac{\sin\theta}{\sin(r-1)\theta}
\]
(see Figure \ref{fig:ZerosOutsideDiskrhoeq1}). We claim that the
right side of the above inequality is an increasing function of $\theta$.
Indeed, computing its derivative we obtain 
\[
\frac{\cos\theta\sin(r-1)\theta-(r-1)\sin\theta\cos(r-1)\theta}{\sin^{2}(r-1)\theta}=\frac{r\sin(r-2)\theta-(r-2)\sin r\theta}{2\sin^{2}(r-1)\theta},
\]
whose numerator is nonnegative, since it vanishes when $\theta=0$
and it is non-decreasing. Thus 
\[
\frac{\sin\theta}{\sin(r-1)\theta}\le\sin\frac{\pi}{2(r-1)}\le\sin\frac{\pi}{r},
\]
and consequently $\theta^{t^{*}}\leq\theta_{\max}^{t^{*}}<\pi/r$.\\
 Given now that $\theta^{t^{*}}\in(0,\pi/r)$, equations \eqref{eq:proddistances},
\eqref{eq:theta_kdef}, and \eqref{eq:sumthetak} imply that for some
$0\le l<n$ the equality 
\[
\sum_{k=1}^{n}\theta_{k}(\theta^{t^{*}};l)=r\theta^{t^{*}}+l\pi,
\]
holds, along with $z(\theta^{t^{*}};l)=z(\theta)$. The latter equality
implies in particular that $z(\theta^{t^{*}};l)$ is positive, and
therefore $n-l-1$ must be even. Consequently, $z(\theta;l)$ is monotone
increasing in $\theta$. Observe that 
\[
\tau(\theta^{t^{*}};l)\leq\tau\leq\tau(\theta;l),
\]
where the first inequality follows from $t^{*}\in C_{1}\cap C_{2}$,
and the second inequality follows from Remark \ref{rem:anglesumdecrease}.
By continuity, there is and angle $\widetilde{\theta}$ between $\theta$
and $\theta^{t^{*}}$ so that $\tau=\tau(\widetilde{\theta};l)$.
It follows that $\widetilde{\theta}\ge\theta$ if and only if 
\[
z(\widetilde{\theta};l)=\frac{\prod|\tau e^{i\widetilde{\theta}}-\tau_{k}|}{|\tau|^{r}}\ge\frac{\prod|\tau e^{i\theta}-\tau_{k}|}{|\tau|^{r}}=z(\theta).
\]
On the other hand, $z(\theta;l)$ is monotone increasing, and hence
$\widetilde{\theta}\ge\theta^{t^{*}}$ if and only if $z(\widetilde{\theta};l)\ge z(\theta^{t^{*}};l)=z(\theta)$.
We conclude that in fact $\widetilde{\theta}=\theta^{t^{*}}=\theta$,
and therefore $t^{*}=\tau e^{i\theta}$, since this is the only point
in $C_{1}\cap C_{2}$ with argument $\theta$. The proof is complete. 
\end{proof}
With $t_{0,1}=\tau e^{\pm i\theta}$, $q_{k}=t_{k}/t_{0}$ and $\zeta_{k}=q_{k}e^{-i\theta}$,
$0\le k<\max\left\{ n,r\right\} ,$ the following proposition is equivalent
to Proposition \ref{prop:zerodenom}. 
\begin{prop}
\label{prop:zerosQ}If $0<\theta<\pi/r$, then besides the two trivial
zeros $\zeta_{0,1}=e^{\pm i\theta}$, all the zeros $\zeta_{k}$,
$2\le k<\max\left\{ n,r\right\} $, of the polynomial 
\begin{equation}
Q(\zeta)=\prod_{j=1}^{n}\left(\frac{\sin(\theta_{j}-\theta)}{\sin\theta}-\zeta\frac{\sin\theta_{j}}{\sin\theta}\right)+\zeta^{r},\label{eq:Qzeta}
\end{equation}
lie outside the closed unit disk. 
\end{prop}
We focus on the values of $\theta$ where $\cos(m+r)\theta=\pm1$
and $0<\theta<\pi/r$, i.e., 
\begin{equation}
\theta=\frac{h\pi}{m+r}\label{eq:thetaform}
\end{equation}
where $h=1,\ldots,\left\lfloor m/r\right\rfloor $. We will show in
Proposition \ref{prop:signchangegeneral} that the sign of $H(\pi/r^{-})$
is $(-1)^{\left\lfloor m/r\right\rfloor +1}$, and that at the values
of $\theta$ given in (\ref{eq:thetaform}), the sign of $H(\theta)$
is $(-1)^{h}$ when $m$ is large. Assuming this fact, the proof of
Theorem \ref{maintheorem} is now simple. By the Intermediate Value
Theorem, $H(\theta)$ has at least $\left\lfloor m/r\right\rfloor $
solutions $\theta$, each of which yields a real solution $z$ of
$H_{m}(z)$ on $(0,\infty)$ by Lemma \ref{lem:zfunctheta}. Theorem
\ref{maintheorem} follows from the fact that the degree of $H_{m}(z)$
is at most $\left\lfloor m/r\right\rfloor $. 
\begin{prop}
\label{prop:signchangegeneral}Suppose $n,r\in\mathbb{N}$ and $\theta$
is given in \eqref{eq:thetaform}. Then 
\begin{itemize}
\item[(i)] $\textrm{sgn}\left(H(\theta)\right)=(-1)^{h}$, and 
\item[(ii)] $\textrm{sgn}\left(H(\pi/r^{-})\right)=(-1)^{\left\lfloor m/r\right\rfloor +1}$ 
\end{itemize}
for all $m$ sufficiently large. 
\end{prop}
To prove this proposition, for large $m$ we consider three cases
when $\theta$ is bounded away from both $0$ and $\pi/r$, when $\theta$
approaches $0$, and when $\theta$ approaches $\pi/r$.

\subsection*{Case 1: $\gamma<\theta<\pi/r-\gamma$ for some fixed small $\gamma$}

Proposition \ref{prop:zerosQ} implies that if $2\le k<\max\left\{ n,r\right\} $
then $|\zeta_{k}|>1+\epsilon$ for some fixed $\epsilon$. Thus in
\eqref{eq:Htheta} the sum 
\[
\sum_{k=2}^{\max\left\{ n,r\right\} }\frac{1}{\zeta_{k}^{m+1}Q'(\zeta_{k})}
\]
approaches $0$ exponentially fast when $m$ is large. The sign of
$H(\theta)$ is then determined by the sum of the first two terms
given in \eqref{eq:mainterm} if this sum does not approach $0$.
Since $A(\theta)>0$ by Lemma \ref{lem:Athetapositive}, the sign
of $H(\theta)$ when $\theta=h\pi/(m+r)$ is $(-1)^{h}$.

\subsection*{Case 2: $\theta\rightarrow0$ as $m\rightarrow\infty$}

We will show that when $\rho=1$ we still have $|\zeta_{k}|>1+\epsilon$
when $2\le k<\max\left\{ n,r\right\} $ and thus arguments in Case
1 apply since $A(\theta)$ approaches $0$ with a polynomial rate
by Lemma \ref{lem:AthetaAsymp}. When $\theta\rightarrow0$, the polynomial
$P(t)+zt^{r}$ approaches $P(t)+at^{r}$ with a real multiple zero
at $t_{a}>\tau_{1}$. We need to show that besides the double zero
at $t_{a}$, all the zeros of $P(t)+at^{r}$ lie outside the closed
disk centered at the origin with radius $t_{a}$. From Proposition
\ref{prop:zerodenom}, it suffices to show that besides the double
zero at $t_{a}$, the moduli of the others zeros are not $t_{a}$.
The fact that these moduli are not $t_{a}$ follow directly from the
inequality 
\[
\frac{\prod_{k=1}^{n}|t-\tau_{k}|}{|t^{r}|}\ge\frac{\prod_{k=1}^{n}|t_{a}-\tau_{k}|}{t_{a}^{r}}
\]
whenever $|t|=|t_{a}|$ and the equality holds only when $t=t_{a}$.
Since the multiplicity of $t_{a}$ is at least $2$, we have that
$P(t_{a})+at_{a}^{r}=0$ and $P'(t_{a})+rat_{a}^{r-1}=0$. These two
equations imply that $rP(t_{a})-t_{a}P'(t_{a})=0$. If the multiplicity
of $t_{a}$ is higher than $2$ then $P''(t_{a})+r(r-1)at_{a}^{r-2}=0$
which can be combined with $P'(t_{a})+rat_{a}^{r-1}=0$ to give $(r-1)P'(t_{a})-t_{a}P''(t_{a})=0$.
By the interlacing property of the zeros of $P'(t)$ and $P''(t)$,
$t_{a}$ must lie between the two smallest zeros of $P'(t)$ and $P''(t)$.
Similarly the equation $rP(t_{a})-t_{a}P'(t_{a})=0$ implies that
$t_{a}$ lies between the two smallest zeros of $P(t)$ and $P'(t)$.
This only occurs when $P'(t_{a})=P''(t_{a})=0$ which then implies
that $P(t_{a})=0$, a contradiction to $\rho=1$.

Next, we consider the case when $\rho>1$. Since $\sum_{j=1}^{n}\theta_{j}=(n-1)\pi+r\theta$,
we have $\theta_{j}\rightarrow\pi$ when $j>\rho$ and $\theta_{1}=\theta_{2}=\cdots=\theta_{\rho}\rightarrow\pi-\pi/\rho$,
see Figure \ref{fig:ZerosOutsideDiskbigrho}. Let $\eta_{j}=\pi-\pi/\rho-\theta_{j}$
for $1\le j\le\rho$ and $\eta_{j}=\pi-\theta_{j}$ if $j\ge\rho$.
Recall that $\theta_{j}$, $1\le j\le n$, satisfy 
\[
\tau_{1}\frac{\sin\theta_{1}}{\sin(\theta_{1}-\theta)}=\cdots=\tau_{\rho+1}\frac{\sin\theta_{\rho+1}}{\sin(\theta_{\rho+1}-\theta)}=\cdots=\tau_{n}\frac{\sin\theta_{n}}{\sin(\theta_{n}-\theta)}.
\]
We have 
\begin{eqnarray*}
\frac{\sin\theta_{1}}{\sin(\theta_{1}-\theta)} & = & \frac{\sin(\pi/\rho+\eta_{1})}{\sin(\pi/\rho+\eta_{1}+\theta)}\\
 & = & \frac{\sin(\pi/\rho)+\cos(\pi/\rho)\eta_{1}+\mathcal{O}(\eta_{1}^{2})}{\sin(\pi/\rho)+\cos(\pi/\rho)(\eta_{1}+\theta)+\mathcal{O}(\eta_{1}^{2}+\eta_{1}\theta+\theta^{2})}\\
 & = & 1-\cot\frac{\pi}{\rho}\theta+\mathcal{O}(\eta_{1}^{2}+\eta_{1}\theta+\theta^{2})
\end{eqnarray*}
and the corresponding fraction when $j>\rho$ is 
\[
\frac{\sin\theta_{j}}{\sin(\theta_{j}-\theta)}=\frac{\eta_{j}+\mathcal{O}(\eta_{j}^{3})}{\eta_{j}+\theta+\mathcal{O}((\eta_{j}+\theta)^{3})}=\frac{\eta_{j}}{\eta_{j}+\theta}\left(1+\mathcal{O}(\eta_{j}^{2}+\theta^{2}+\eta_{j}\theta)\right).
\]
The identity $\tau_{1}\sin\theta_{1}/\sin(\theta_{1}-\theta)=\tau_{j}\sin\theta_{j}/\sin(\theta_{j}-\theta)$
gives 
\[
(\eta_{j}+\theta)\tau_{1}-\tau_{1}\cot\frac{\pi}{\rho}\theta(\eta_{j}+\theta)=\tau_{j}\eta_{j}+\mathcal{O}((\eta_{1}+\eta_{j}+\theta)^{3})
\]
from which we solve for $\eta_{j}$, $j>\rho$ 
\begin{eqnarray*}
\eta_{j} & = & \frac{\tau_{1}\left(\theta-\cot(\pi/\rho)\theta^{2}\right)}{\tau_{j}-\tau_{1}+\tau_{1}\cot(\pi/\rho)\theta}+\mathcal{O}((\eta_{1}+\theta)^{3})\\
 & = & \left(\frac{\tau_{1}}{\tau_{j}-\tau_{1}}\theta-\tau_{1}\frac{\cot(\pi/\rho)\theta^{2}}{\tau_{j}-\tau_{1}}\right)\left(1-\frac{\tau_{1}\cot(\pi/\rho)\theta}{\tau_{j}-\tau_{1}}\right)+\mathcal{O}((\eta_{1}+\theta)^{3})\\
 & = & \frac{\tau_{1}}{\tau_{j}-\tau_{1}}\theta-\frac{\cot(\pi/\rho)\tau_{1}\tau_{j}}{(\tau_{j}-\tau_{1})^{2}}\theta^{2}+\mathcal{O}((\eta_{1}+\theta)^{3}).
\end{eqnarray*}
The equation $\sum_{j=1}^{n}\eta_{j}=-r\theta$ implies that 
\[
\rho\eta_{1}+\left(\sum_{j>\rho}\frac{\tau_{1}}{\tau_{j}-\tau_{1}}+r\right)\theta-\sum_{k>\rho}\frac{\cot(\pi/\rho)\tau_{1}\tau_{j}}{(\tau_{j}-\tau_{1})^{2}}\theta^{2}+\mathcal{O}(\theta^{3})=0.
\]
We recall that if $\zeta$ is a zero of $Q(\zeta)$ then 
\[
\prod_{j=1}^{n}\left(\frac{\sin(\theta_{j}-\theta)}{\sin\theta_{j}}-\zeta\right)+\zeta^{r}\prod_{j=1}^{n}\frac{\sin\theta}{\sin\theta_{j}}=0.
\]
When $\theta\rightarrow0$ the left side approaches the polynomial
$(1-\zeta)^{\rho}\prod_{j>\rho}(\tau_{j}/\tau_{1}-\zeta)$ and thus
it suffices to consider the sum 
\[
\sum_{k=0}^{\rho-1}\frac{1}{\zeta_{k}^{m+1}Q'(\zeta_{k})}
\]
where $\zeta_{0},\ldots,\zeta_{\rho-1}$ approach $1$. Let $\zeta=1+\epsilon$,
$\epsilon\in\mathbb{C}$. If $\zeta$ is a zero of $Q(\zeta)$ then
\[
\left(\cot\frac{\pi}{\rho}\theta-\epsilon\right)^{\rho}\prod_{j>\rho}\left(\frac{\tau_{j}}{\tau_{1}}-1\right)(1+\mathcal{O}(\theta+\epsilon))+\frac{\theta^{\rho}}{\sin(\pi/\rho)^{\rho}}\prod_{j>\rho}\frac{\tau_{j}-\tau_{1}}{\tau_{1}}(1+\mathcal{O}(\theta+\epsilon))=0.
\]
We cancel the common factor and obtain 
\[
\theta\cot\frac{\pi}{\rho}-\epsilon=\omega_{k}\frac{\theta}{\sin(\pi/\rho)}+\mathcal{O}((\theta+\epsilon)^{2})
\]
where $\omega_{k}=e^{(2k-1)\pi i/n}$, $0\le k<n$, are the $n$-th
root of $-1$. We solve for $\epsilon$ 
\[
\epsilon=\frac{\cos(\pi/\rho)-\omega_{k}}{\sin(\pi/\rho)}\theta+\mathcal{O}(\theta^{2}).
\]
Then \eqref{eq:Qprimeatzero} gives 
\begin{eqnarray*}
Q'(\zeta_{k}) & = & \sum_{j=1}^{n}\frac{\zeta_{k}^{r}\sin\theta_{j}}{\sin(\theta_{j}-\theta)-\zeta_{k}\sin\theta_{j}}+r\zeta_{k}^{r-1}\\
 & = & \zeta_{k}^{r-1}\left(\frac{\rho\sin(\pi/\rho)}{\omega_{k}\theta}+\mathcal{O}(1)\right).
\end{eqnarray*}
and consequently 
\[
\sum_{k=0}^{\rho-1}\frac{1}{\zeta_{k}^{m+1}Q'(\zeta_{k})}=\frac{\theta}{\rho\sin(\pi/\rho)}\sum_{k=0}^{\rho-1}\frac{\omega_{k}}{\zeta_{k}^{m+r}}\left(1+\mathcal{O}(\theta)\right)
\]
whose sign is $(-1)^{h}$ by (35) of \cite{ft}.

\subsection*{Case 3: $\theta\rightarrow\pi/r$ as $m\rightarrow\infty$}

Similar to arguments in the beginning of Case 2, when $r=1$, we claim
that $|\zeta_{k}|>1+\epsilon$ for $2\le k<n$ and with Lemma \ref{lem:AthetaAsymp}
the arguments in Case 1 apply. We prove this claim by showing that
besides the double zero at $t_{b}$, all the other zeros of $P(t)+bt$
lie outside the circle radius $|t_{b}|$ or equivalently (with Proposition
\ref{prop:zerodenom}) their moduli are different from $|t_{b}|$.
The claim about their moduli again follows from the inequality 
\[
\frac{\prod_{k=1}^{n}|t-\tau_{k}|}{|t^{r}|}\le\frac{\prod_{k=1}^{n}|t_{b}-\tau_{k}|}{t_{b}^{r}}
\]
when $|t|=|t_{b}|$ with the equality holds only when $t=t_{b}$.
If the multiplicity of $t_{b}$ is more than $2$ then $P''(t_{b})=0$.
This is a contradiction since the zeros of $P''(t)$ are positive
while $t_{b}\le0$ by its definition.

Next, we consider when $r>1$. Since $\sum_{j=1}^{n}\theta_{j}=r\theta+(n-1)\pi$,
we must have $\theta_{j}$, $1\le j\le n$, approach $\pi$. Let $\eta_{j}=\pi-\theta_{j}$,
$1\le j\le n$, and $\eta=\pi/r-\theta$. From the equations 
\begin{eqnarray*}
\tau & = & \tau_{j}\frac{\sin\theta_{j}}{\sin(\theta_{j}-\theta)}\\
 & = & \tau_{j}\frac{\eta_{j}+\mathcal{O}(\eta_{j}^{3})}{\sin(\pi/r)+\cos(\pi/r)(\eta_{j}-\eta)+\mathcal{O}((\eta_{j}+\eta)^{2})}\\
 & = & \frac{\tau_{j}\eta_{j}}{\sin(\pi/r)}\left(1-\cot\frac{\pi}{r}(\eta_{j}-\eta)+\mathcal{O}((\eta_{j}+\eta)^{2})\right)
\end{eqnarray*}
and $\sum_{j=1}^{n}\eta_{j}=r\eta$ we have 
\[
\tau\sin\frac{\pi}{r}\sum_{j=1}^{n}\frac{1}{\tau_{j}}=r\eta(1+\mathcal{O}(\eta_{j}+\eta)).
\]
Thus for any $1\le k\le n$ 
\[
\eta_{k}\tau_{k}\sum_{j=1}^{n}\frac{1}{\tau_{j}}=r\eta(1+\mathcal{O}(\eta)).
\]
The formula \eqref{eq:Qzeta} implies that besides $r$ zeros $\zeta_{0},\ldots,\zeta_{r-1}$
of $Q(\zeta)$ which approach the zeros of $1+\zeta^{r}$, the possible
remaining $\max\left\{ n,r\right\} -r$ zeros approach $\infty$.
Thus it suffices to consider the sign of 
\[
\sum_{k=0}^{r-1}\frac{1}{\zeta_{k}^{m+1}Q'(\zeta_{k})}.
\]
With $\zeta_{k}=e_{k}+\epsilon$, $\epsilon\in\mathbb{C}$, $e_{k}=e^{(2k-1)\pi i/r}$,
the equation $Q(\zeta_{k})=0$ gives 
\begin{eqnarray*}
0 & = & \prod_{j=1}^{n}\left(\frac{\sin(\theta_{j}-\theta)}{\sin\theta}-\zeta_{k}\frac{\sin\theta_{j}}{\sin\theta}\right)+\zeta_{k}^{r}\\
 & = & \prod_{j=1}^{n}\left(\frac{\sin(\pi/r)+\cos(\pi/r)(\eta_{j}-\eta)+\mathcal{O}(\eta^{2})}{\sin(\pi/r)-\cos(\pi/r)\eta+\mathcal{O}(\eta^{2})}-\zeta_{k}\frac{\eta_{j}+\mathcal{O}(\eta^{3})}{\sin(\pi/r)-\cos(\pi/r)\eta+\mathcal{O}(\eta^{2})}\right)+\zeta_{k}^{r}\\
 & = & \prod_{j=1}^{n}\left(1+\cot\frac{\pi}{r}\eta_{j}-e_{k}\frac{\eta_{j}}{\sin(\pi/r)}+\mathcal{O}(\eta^{2}+\epsilon\eta)\right)+(e_{k}+\epsilon)^{r}\\
 & = & \frac{\cos(\pi/r)-e_{k}}{\sin(\pi/r)}\sum_{j=1}^{n}\eta_{j}+\frac{r\epsilon}{e_{k}}+\mathcal{O}\left(\eta^{2}+\epsilon\eta+\eta^{2}\right)
\end{eqnarray*}
from which we solve for $\epsilon$ 
\[
\epsilon=\frac{e_{k}(\cos(\pi/r)-e_{k})\eta}{\sin(\pi/r)}+\mathcal{O}(\eta^{2}).
\]
From \eqref{eq:Qprimeatzero}, we obtain 
\begin{eqnarray*}
Q'(\zeta_{k}) & = & \sum_{j=1}^{n}\frac{\zeta_{k}^{r}\sin\theta_{j}}{\sin(\theta_{j}-\theta)-\zeta_{k}\sin\theta_{j}}+r\zeta_{k}^{r-1}\\
 & = & r\zeta_{k}^{r-1}+\mathcal{O}(\eta)
\end{eqnarray*}
and thus the sign of 
\[
\sum_{k=0}^{r-1}\frac{1}{\zeta_{k}^{m+1}Q'(\zeta_{k})}=\frac{1}{r}\sum_{k=0}^{r-1}\frac{1+\mathcal{O}(\eta)}{\zeta_{k}^{m+r}}
\]
is $(-1)^{h}$ by (41) of \cite{ft}.

\section{Open problems}

Theorem \ref{maintheorem} asserts that there is a constant $C$,
depending on both $P(t)$ and $r$, such that the zeros of $H_{m}(z)$
lie in the real interval $(a,b)\subseteq(0,\infty)$ for all $m\ge C$.
We conjecture that the zeros of $H_{m}(z)$ lie in the interval $(a,b)$
for all $m$. 
\begin{conjecture}
Suppose $P(t)$ is a real polynomial whose zeros are positive real
numbers and $P(0)>0$. Let $r$ be a positive integer such that $\max\{\deg P,r\}>1$.
For all integers $m$, the zeros of the polynomial $H_{m}(z)$ generated
by 
\[
\sum_{m=0}^{\infty}H_{m}(z)t^{m}=\frac{1}{P(t)+zt^{r}}
\]
lie in the interval $(a,b)$ defined in Theorem \ref{maintheorem}.
Moreover, the set $\mathcal{Z}=\bigcup_{m\gg1}\{z\ |\ H_{m}(z)=0\}$
is dense on $(a,b)$. 
\end{conjecture}
With a different approach, in \cite[Theorem 1]{err} the authors prove
that the zeros of the polynomial $H_{m}(z)$ generated by 
\[
\sum_{m=0}^{\infty}H_{m}(z)t^{m}=\frac{1}{1-zt+Ct^{2}+t^{3}}
\]
are real for all $m\ge1$ if $C\ge3$. On the other hand, if $C<3$,
then there exists an $m\ge1$ so that not all zeros of $H_{m}(z)$
are real (see \cite[Proposition 1]{bbs}). In a similar vein, \textit{The
6-Conjecture} \cite[Conjecture 2]{err} stipulates that the zeros
of the sequence of polynomials generated by 
\[
\sum_{m=0}^{\infty}H_{m}(z)t^{m}=\frac{1}{1-zt+Ct^{2}-4t^{3}+t^{4}}.
\]
are real if $C\ge6$. The authors of \cite{err} believe that a similar conclusion still
holds if the coefficients of the denominator follow a binomial pattern.
In light of Theorem \ref{maintheorem}, we conjecture a more general
result, by replacing the binomial polynomial by a polynomial, whose
zeros are real and of the same sign. 
\begin{conjecture}
Let $C$ be a real number and $r$,$s$ be natural numbers. Assume
that $P(t)$ is a polynomial whose zeros are real and positive. If
$C(s-r)\ge0$, then the zeros of the sequence of polynomials $H_{m}(z)$
generated by 
\[
\sum_{m=0}^{\infty}H_{m}(z)t^{m}=\frac{1}{P(t)+Ct^{s}+zt^{r}}
\]
are real. 
\end{conjecture}

\end{document}